\documentclass[12pt]{amsart}
\usepackage{amsmath,amsthm,enumerate,graphics,amsfonts,latexsym,amsopn,verbatim,amscd,amssymb}
\usepackage{color}
\usepackage{tikz}
\usepackage{amsmath,amssymb,graphicx,tikz,amsthm}
\usetikzlibrary{arrows,matrix}
\usepackage{url}

\theoremstyle{plain}
\newtheorem{thm}{Theorem}[section]

\newtheorem{cor}[thm]{Corollary}

\newtheorem{conj}[thm]{Conjecture}

\newtheorem{problem}{Theorem}[section] 
\newtheorem{pprob}[problem]{Problem}
\newenvironment{prob}{\pprob\rm}{\endpprob}

\theoremstyle{definition}

\DeclareMathOperator{\snm}{snm}
\DeclareMathOperator{\E}{E}
\DeclareMathOperator{\sn}{sn}

\DeclareMathOperator{\cnspec}{CN-spec}

\DeclareMathOperator{\CN}{CN}

\DeclareMathOperator{\Ecn}{E_{CN}}

\tikzset{every path/.style=thick,
	acteur/.style={
		circle,
		fill=red,
		thick,
		inner sep=1pt,
		minimum size=0.15cm
}}

\usepackage{ragged2e}

\begin{document}

\title[MSN-energy of commuting graphs of  finite rings]{Minimum second neighborhood degree energy of commuting graphs of finite rings}

\author[P. Tak, J. Dutta and R. K. Nath]{Payal Tak, Jutirekha Dutta and Rajat  Kanti Nath$^*$}

\address{P. Tak, \, Department of Mathematical Science, \, Tezpur  University, Napaam -784028, Sonitpur, Assam, India.}
\email{payaltak111@gmail.com}

\address{J. Dutta, Tezpur  University Campus, Napaam -784028, Sonitpur, Assam, India.}
\email{jutirekhadutta@yahoo.com}

\address{Rajat  Kanti Nath, \, Department of Mathematical Science, \, Tezpur  University, Napaam -784028, Sonitpur, Assam, India.}

\email{rajatkantinath@yahoo.com}

\begin{abstract}
In this paper, we compute minimum second neighborhood degree spectrum and energy of commuting graphs of certain finite non-commutative  rings. In particular, we consider non-commutative  rings of order $p^2, p^3, p^4, p^5, p^2q$ and $p^3q$, where $p$ and $q$ are primes. We shall also show that the commuting graphs of these rings are MSN-integral but not MSN-hyperintegral. Finally, employing the techniques used in this paper, we prove \cite[Conjecture 3]{NFDS-2021} and \cite[Conjecture 3.12]{fn24}. We conclude this paper with two open problems.
\end{abstract}

\thanks{ $^*$Corresponding author}

\maketitle

\section{Introduction}
Let $\Gamma$ be a finite simple graph with vertex set $V(\Gamma) = \{v_1, v_2, \dots, v_n\}$. For any vertex $v_i \in V(\Gamma)$, we write $N(v_i)$ to denote the neighborhood of $v_i$ which is given by the set $\{v_j \in V(\Gamma) : v_j \sim v_i\}$ (where $v_j \sim v_i$ means $v_j$ is adjacent to $v_i$). The second neighborhood of $v_i$ is denoted by $N^2(v_i)$ and it is defined as $N^2(v_i) = \cup_{u \in N(v_i)}N(u) \setminus \{v_i\}$. Let $\delta_2(v_i)$ be the second neighborhood degree of $v_i$. Then $\delta_2(v_i) = \sum_{x \in N^2(v_i)}\deg x$, where $\deg x$ is the degree of $x$ in $\Gamma$. The minimum second neighborhood degree matrix of $\Gamma$, denoted by $\snm(\Gamma)$, is defined as 
\[
\snm(\Gamma)_{i, j} = \begin{cases}
\min\{\delta_2(v_i), \delta_2(v_j)\}, &\text{if $i \ne j$ and $v_i \sim v_j$}\\
0 , & \text{ otherwise},\\
\end{cases}
\]
where $\snm(\Gamma)_{i, j}$ is the $(i, j)$-th element of $\snm(\Gamma)$. The set of eigenvalues of $\snm(\Gamma)$ with multiplicities is called the minimum second neighborhood degree spectrum (for short MSN-spectrum) of $\Gamma$ and it is denoted by $\sigma_{\snm}(\Gamma)$. We write $\sigma_{\snm}(\Gamma) = \big{\{}[\sigma_1]^{\alpha_1}, [\sigma_2]^{\alpha_2},$ $ \dots, [\sigma_m]^{\alpha_m}\big{\}}$ to denote that $\sigma_1, \sigma_2, \dots, \sigma_m$ are the distinct eigenvalues of $\snm(\Gamma)$ having multiplicities $\alpha_1, \alpha_2, \dots, \alpha_m$ respectively. A graph $\Gamma$ is called MSN-integral if $\sigma_i$  is an integer for $i = 1, 2, \dots, m$.

The minimum second neighborhood degree energy (for short MSN-energy) of $\Gamma$, denoted by $\E_{\sn}(\Gamma)$, is defined as
\[
\E_{\sn}(\Gamma) = \sum_{i = 1}^{m}\alpha_i|\sigma_i|.
\] 
Note that $\snm(K_n) = (n-1)^2A(K_n)$, where $A(K_n)$ is the adjacency matrix of $K_n$. Therefore, $\sigma_{\snm}(K_n) = \{[-(n-1)^2]^{n - 1}, [(n-1)^3]^1\}$ and so $\E_{\sn}(K_n) = 2(n - 1)^3$. Further, if $\Gamma = l_1K_{m_1} \cup l_2K_{m_2} \cup \cdots \cup l_rK_{m_r}$ then 
\begin{align}\label{Eq-1}
	\sigma_{\snm}(\Gamma) = &\Big{\{} [-(m_1-1)^2]^{l_1(m_1 - 1)},  [(m_1-1)^3]^{l_1}, [-(m_2-1)^2]^{l_2(m_2 - 1)}, \nonumber\\
	&\qquad\qquad\qquad[(m_2-1)^3]^{l_2}, \dots, [-(m_r-1)^2]^{l_r(m_r - 1)}, [(m_r-1)^3]^{l_r}\Big{\}}
\end{align}
and 
\begin{equation}\label{Eq-2}
	\E_{\sn}(\Gamma) = 2\sum_{i = 1}^r l_i(m_i - 1)^3.
\end{equation}

A graph $\Gamma$ is called minimum second neighborhood  
 hyperenergetic (for short MSN-hyperenergetic) if $\E_{\sn}(\Gamma) > \E_{\sn}(K_n)$, where $K_n$ is the complete graph on $n$-vertices. Motivated by the notions of graph energy  and common neighborhood energy abbreviated as CN-energy  (pioneered by Gutman \cite{Gutman-1978,ASG}),  Manilal and Harikrishan \cite{MH24-CMI} introduced the notion of MSN-energy. A survey on various graph energies and their applications can be found in \cite{GF-2019}.  The concepts of MSN-hyperenergetic and  MSN-integral graphs are analogous to hyperenergetic \cite{Gutman1999, Walikar} and  integral \cite{hS74}   graphs. In \cite{HM-2024}, Harikrishan and Manilal  computed MSN-energy of commuting graphs of various finite non-abelian groups whose central quotient is isomorphic to a dihedral group, Suzuki group or group of order $p^2$ (where $p$ is a prime) including other classes of groups and also showed that commuting graphs of these groups are not MSN-hyperenergetic. It is worth mentioning that the commuting  graphs of these groups are also considered in \cite{DN1, DN2, SND-PJM-2022} (to show that they are integral and hyperenergetic) and \cite{FSN-2021, NFDS-2021} (to show that they are CN-integral and CN-hyperenergetic). The commuting graph of a finite non-abelian group $G$ is a graph whose vertex set is the set of non-central elements of $G$ and two distinct vertices are adjacent if they commute. This graph is studied widely (see \cite{DasN-2016,DBN-KJM-2020,DN3,DN-IJPAM-2021,GP-2013, IJ-2008,MP-2013,Nath-2018,N-2024,Parker-2013} for example) since the work of Brauer and Fowler \cite{BF-55}. In this paper, we consider commuting graphs of finite non-commutative rings and compute their MSN-spectrum and MSN-energy. As a consequence, we determine finite non-commutative rings such that their commuting graphs are MSN-integral and MSN-hyperenergetic.
 
Let $R$ be a finite non-commutative ring with center $Z(R)$. We write $\Gamma_R$ to denote the commuting graph of $R$ which is a graph with vertex set $R \setminus Z(R)$ and two distinct vertices are adjacent if they commute. The study of commuting graphs of  (semisimple) rings was initiated by Akbari et al. \cite{aghm04}. Abdollahi \cite{a08} and Mohammadian \cite{m10} considered commuting graphs of certain finite matrix rings in their study. It is noteworthy that Omidi and Vatandoost \cite{ov11} carried their study for any finite non-commutative rings. 
In Section 2, we consider finite rings whose central quotient (as additive group) is isomorphic to  ${\mathbb{Z}}_p \times {\mathbb{Z}}_p$ for any prime $p$.
In Section 3, we consider non-commutative rings of order $p^4$ and $p^5$ for any prime $p$. In Section 4, we consider non-commutative rings of order $p^2q$ and $p^3q$ for any primes $p$ and $q$. Finally, we conclude this paper with Section 5 where non-commutative CC-rings (a class of rings introduced by Erfanian et al. \cite{ekn15}) are considered.

For the graph $\Gamma$ with vertex set $V(\Gamma) = \{v_1, v_2, \dots, v_n\}$, the common neighborhood matrix of $\Gamma$ (denoted by $\CN(\Gamma)$) is given by
\[
\CN(\Gamma)_{i, j} = \begin{cases}
|C(v_i, v_j)|, & \text{ if } i \ne j\\
0, & \text{ if } i = j;
\end{cases}
\]
where $C(v_i, v_j)$ is the common neighborhood of $v_i$ and $v_j$ (see \cite{ASG}). Let $\cnspec(\Gamma)$ be the set of eigenvalues of $\CN(\Gamma)$ with multiplicities and $\Ecn(\Gamma)$ be the CN-energy of $\Gamma$. Then  $\Ecn(\Gamma) = \sum_{\sigma \in \cnspec(\Gamma)}|\sigma|$.  We have the following two conjectures regarding CN-energy of graphs.
\begin{conj}\label{Conj-1.1}\cite[Conjecture 3]{NFDS-2021} 
If $\Gamma = l_1K_{m_1} \cup l_2K_{m_2} \cup \cdots \cup l_rK_{m_r}$ then  $\Gamma$ is not CN-hyperenergetic.
\end{conj} 
\begin{conj}\label{Conj-1.2}\cite[Conjecture 3.12]{fn24}
If $R$ is a finite CC-ring then $\Gamma_R$ is not CN-hyperenergetic.
\end{conj}
We employ the techniques used in this paper and prove Conjecture \ref{Conj-1.1} and Conjecture \ref{Conj-1.2} in Section 5.

\section{Rings whose central quotient is isomorphic to ${\mathbb{Z}}_p \times {\mathbb{Z}}_p$}
This class of rings are considered in \cite{Dutta-Nath-CCR,fn24,N20} in order to compute various spectrum and energies of the commuting graphs. In \cite[Theorem 2.4]{Dutta-Nath-CCR}, it was shown that $\Gamma_R =  (p + 1) K_{(p - 1)m}$ if $\frac{R}{Z(R)}$ is isomorphic to ${\mathbb{Z}}_p \times {\mathbb{Z}}_p$, where $p$ is a prime and $|Z(R)| = m$. Therefore, using \eqref{Eq-1} and \eqref{Eq-2} we get the following result.
\begin{thm}\label{ZpxZp}
If $\frac{R}{Z(R)} \cong {\mathbb{Z}}_p \times {\mathbb{Z}}_p$ and $|Z(R)| = m$ then 
\[
\sigma_{\snm}(\Gamma_R) = \Big{\{}[-((p - 1)m-1)^2]^{(p + 1)((p - 1)m - 1)},  [((p - 1)m-1)^3]^{p + 1}  \Big{\}}
\] 
and 
$
\E_{\sn}(\Gamma_R) = 2(p + 1)((p - 1)m - 1)^3.
$
\end{thm}

We know that finite $n$-centralizer rings (for certain values of $n$) were characterized in terms of their central quotients (see \cite{dn16,dbn22,dbn23} for details). The notion of   $n$-centralizer ring is analogous to the notion of $n$-centralizer group which was introduced by Belcastro and Sherman \cite{BF-55} in 1994. In the following corollary, we compute MSN-spectrum and MSN-energy of commuting graphs of certain finite $n$-centralizer rings.
\begin{cor}\label{various-4-cent}
Let  $R$ be a finite $n$-centralizer ring and $|Z(R)| = m$.
\begin{enumerate}
\item If \, $n = 4$ \, then \, $
\sigma_{\snm}(\Gamma_R) \,\, = \,\, \Big{\{}[-(m-1)^2]^{3m - 3}, \quad  [(m-1)^3]^{3}  \Big{\}}
$ \,
and 
$
\E_{\sn}(\Gamma_R) = 6(m - 1)^3.
$

\item If \, $n = 5$ \, then \, $
\sigma_{\snm}(\Gamma_R) \, = \, \Big{\{}[-(2m-1)^2]^{8m - 4},  \quad 
[(2m-1)^3]^{4}  \Big{\}}
$ \,
and 
$
\E_{\sn}(\Gamma_R) = 8(2m - 1)^3.
$

\item If \, $n = 7$ \, then \, $\sigma_{\snm}(\Gamma_R) \, = \, \Big{\{}[-(4m-1)^2]^{24m - 6}, \quad [(4m-1)^3]^{6}  \Big{\}}$ 
and  
$
\E_{\sn}(\Gamma_R) = 12(4m - 1)^3.
$		
		
\item If $n = p + 2$ and $|R| = p^k$ for  $k \in \mathbb{N}$ then
\[
\sigma_{\snm}(\Gamma_R) = \Big{\{}[-((p - 1)m-1)^2]^{(p + 1)((p - 1)m - 1)},  [((p - 1)m-1)^3]^{p + 1}  \Big{\}}
\] 
and 
$
\E_{\sn}(\Gamma_R) = 2(p + 1)((p - 1)m - 1)^3.
$
\end{enumerate}
\end{cor}

\begin{proof}
(a)  It was shown in \cite[Theorem 3.1]{dbn22} that  \,  $\frac{R}{Z(R)} \cong {\mathbb{Z}}_2 \times {\mathbb{Z}}_2$ \, when $R$ is a $4$-centralizer ring. Therefore, putting $p = 2$ in Theorem \ref{ZpxZp}, we get
\[
\sigma_{\snm}(\Gamma_R) = \Big{\{}[-(m-1)^2]^{3(m - 1)},  [(m-1)^3]^{3}  \Big{\}}
\] 
and 
$
\E_{\sn}(\Gamma_R) = 2\times 3(m - 1)^3.
$

(b)  It was shown in \cite[Theorem 4.1]{dbn22} that $\frac{R}{Z(R)} \cong {\mathbb{Z}}_3 \times {\mathbb{Z}}_3$ when  $R$ is a $5$-centralizer ring. Therefore, putting $p = 3$ in Theorem \ref{ZpxZp}, we get
\[
\sigma_{\snm}(\Gamma_R) = \Big{\{}[-(2m-1)^2]^{4(2m - 1)},  
[(2m-1)^3]^{4}  \Big{\}}
\] 
and 
$
\E_{\sn}(\Gamma_R) = 2\times 4(2m - 1)^3.
$

(c) It was shown in \cite[Theorem 3.5]{dbn23} that $\frac{R}{Z(R)} \cong {\mathbb{Z}}_5 \times {\mathbb{Z}}_5$ when  $R$ is a $7$-centralizer ring. Therefore, putting $p = 5$ in Theorem \ref{ZpxZp}, we get
\begin{center}
	$\sigma_{\snm}(\Gamma_R) = \Big{\{}[-(4m-1)^2]^{6(4m - 1)},  [(4m-1)^3]^{6}  \Big{\}}$
\end{center} 
and 
$
\E_{\sn}(\Gamma_R) = 2\times 6(4m - 1)^3.
$		

(d) It was shown in \cite[Theorem 2.6]{dbn22} that  $\frac{R}{Z(R)} \cong {\mathbb{Z}}_p \times {\mathbb{Z}}_p$ when  $|R| = p^k$ and $R$ is a $(p + 2)$-centralizer ring. 
Therefore, by Theorem \ref{ZpxZp} the result follows.
\end{proof}

\begin{cor}\label{com-deg-related}
Let  $\Pr(R)$  be the commuting probability of $R$ and $|Z(R)|$ $= m$.
\begin{enumerate}
\item If \, $\Pr(R)  =   \frac{5}{8}$ \,  then  $
\sigma_{\snm}(\Gamma_R) = \Big{\{}[-(m-1)^2]^{3(m - 1)}, \,\, [(m-1)^3]^{3}  \Big{\}}
$
and 
$
\E_{\sn}(\Gamma_R) = 6(m - 1)^3.
$

\item If $\Pr(R) = \frac{p^2 + p - 1}{p^3}$, where $p$ is the smallest prime divisor of   $|R|$, then \[
\sigma_{\snm}(\Gamma_R) = \Big{\{}[-((p - 1)m-1)^2]^{(p + 1)((p - 1)m - 1)},  [((p - 1)m-1)^3]^{p + 1}  \Big{\}}
\] 
and 
$
\E_{\sn}(\Gamma_R) = 2(p + 1)((p - 1)m - 1)^3.
$
\end{enumerate}
\end{cor} 
\begin{proof}
(a) It was shown in \cite[Theorems 1]{dmachale} that $\frac{R}{Z(R)} \cong {\mathbb{Z}}_2 \times {\mathbb{Z}}_2$ if $\Pr(R) = \frac{5}{8}$. Hence, the result follows from Theorem \ref{ZpxZp} by putting $p = 2$.

(b) It was shown in \cite[Theorems 3]{dmachale} that $\frac{R}{Z(R)} \cong {\mathbb{Z}}_p \times {\mathbb{Z}}_p$ if  $\Pr(R) = \frac{p^2 + p - 1}{p^3}$, where $p$ is   the smallest prime divisor of $|R|$. Hence, the result follows from Theorem \ref{ZpxZp}.
\end{proof}

\begin{cor}\label{various-ring-p}
Let $p$ be any prime and $R$ be a non-commutative ring. 
\begin{enumerate}
\item If  $|R| = p^2$ then \[
\sigma_{\snm}(\Gamma_R) = \Big{\{}[-((p - 1)m-1)^2]^{(p + 1)((p - 1)m - 1)},  [((p - 1)m-1)^3]^{p + 1}  \Big{\}}
\] 
and 
$
\E_{\sn}(\Gamma_R) = 2(p + 1)((p - 1)m - 1)^3.
$

\item If $R$ has unity and $|R| = p^3$  then \[
\sigma_{\snm}(\Gamma_R) = \Big{\{}[-((p - 1)m-1)^2]^{(p + 1)((p - 1)m - 1)},  [((p - 1)m-1)^3]^{p + 1}  \Big{\}}
\] 
and 
$
\E_{\sn}(\Gamma_R) = 2(p + 1)((p - 1)m - 1)^3.
$
\end{enumerate}
\end{cor}

\begin{proof}
(a) If $R$ is a non-commutative ring of order $p^2$, for any prime $p$, then  $|Z(R)| = 1$ and so $\frac{R}{Z(R)} \cong {\mathbb{Z}}_p \times {\mathbb{Z}}_p$. Hence, the result follows from Theorem \ref{ZpxZp}.

(b) If $R$ is a non-commutative ring with unity having order $p^3$, for any prime $p$, then  $|Z(R)| = p$ and so $\frac{R}{Z(R)} \cong {\mathbb{Z}}_p \times {\mathbb{Z}}_p$.  Hence, the result follows from Theorem \ref{ZpxZp}.
\end{proof}

\begin{thm}\label{various-cons-1}
If  $\frac{R}{Z(R)} \cong {\mathbb{Z}}_p \times {\mathbb{Z}}_p$ then   $\Gamma_R$ is MSN-integral but not MSN-hyperenergetic. 
\end{thm}
\begin{proof}
Let $|Z(R)| = m$. Then $|V(\Gamma_R)| = |R| - |Z(R)| = (p^2 - 1)m$. Therefore,
\begin{align*}
\E_{\sn}(K_{|V(\Gamma_R)|}) = 2((p^2 - 1)m - 1)^3 &> 2 (p + 1)^3((p - 1)m - 1)^3\\
& > 2 (p + 1)((p - 1)m - 1)^3 = \E_{\sn}(\Gamma_R).
\end{align*}
Hence, $\Gamma_R$ is  not MSN-hyperenergetic. The fact that $\Gamma_R$ is MSN-integral follows from Theorem \ref{ZpxZp}.
\end{proof}
We conclude this section with the following corollary.
\begin{cor}
If $R$ is any ring considered in Corollary \ref{various-4-cent}--Corollary \ref{various-ring-p} then   $\Gamma_R$ is MSN-integral but not MSN-hyperenergetic.
\end{cor}
\section{Rings of order $p^4$ and $p^5$}
The structures of commuting graphs of these two classes of rings are determined in \cite{vr16}. Later on, various spectra and energies of these graphs are computed in \cite{fn24, fns20}. Further, in \cite{fn24-BSPM} genus of commuting graphs of these classes of rings are computed and also determined whether they are planar or toroidal. In this section, we consider non-commutative rings of order $p^4$ and $p^5$, for any prime $p$, and show that the commuting graphs of these rings are MSN-integral but not MSN-hyperenergetic.  
\begin{thm}\label{order-p4}
Let $R$ be a non-commutative ring with unity and $|R| = p^4$.
\begin{enumerate}
\item Suppose that  $Z(R)$ has $p$ elements. Then  

\noindent $\sigma_{\snm}(\Gamma_R)  \, = \, \Big{\{} [-(p^2 - p-1)^2]^{(p^2 + p + 1)(p^2 - p - 1)}, \,\,\,  [(p^2 - p-1)^3]^{p^2 + p + 1}\Big{\}} 
$ \, 
and 
$\E_{\sn}(\Gamma_R) = 2(p^2 + p + 1)(p^2 - p - 1)^3$ 
		or 
$\sigma_{\snm}(\Gamma_R) = \Big{\{} [-(p^2 - p-1)^2]^{l_1(p^2 - p - 1)},$  $[(p^2 - p-1)^3]^{l_1}, \quad [-(p^3 - p-1)^2]^{l_2(p^3 - p - 1)}, \quad
[(p^3 - p-1)^3]^{l_2}\Big{\}}
$ \quad
and 

\noindent $\E_{\sn}(\Gamma_R) = 2 l_1(p^2 - p - 1)^3 + 2l_2(p^3 - p - 1)^3$, where $l_1 + l_2(p + 1) = p^2 + p + 1$.
		
\item Suppose that  $Z(R)$ has  $p^2$ elements. Then 
\begin{center}
	$\sigma_{\snm}(\Gamma_R) = \Big{\{} [-(p^3 - p^2-1)^2]^{(p + 1)(p^3 - p^2 - 1)},  [(p^3 - p^2-1)^3]^{p + 1}\Big{\}}
$
\end{center}
and 
$\E_{\sn}(\Gamma_R) = 2(p + 1)(p^3 - p^2 - 1)^3$.		
\end{enumerate}
\end{thm}
\begin{proof}
(a) It was shown in  \cite[Theorem 2.5]{vr16} that \, $\Gamma_R = (p^2 + p + 1)K_{p(p - 1)}$ or $l_1K_{p(p - 1)}\cup l_2K_{p(p^2 - 1)}$, when $l_1 + l_2(p + 1) = p^2 + p + 1$. If $\Gamma_R = (p^2 + p + 1)K_{p(p - 1)}$  then, by  \eqref{Eq-1} and \eqref{Eq-2}, we have \,
$\sigma_{\snm}(\Gamma_R) \quad = \quad  \Big{\{} [-(p(p - 1)-1)^2]^{(p^2 + p + 1)(p(p - 1) - 1)},$ $  [(p(p - 1) - 1)^3]^{p^2 + p + 1}\Big{\}}
$
and 
$\E_{\sn}(\Gamma_R) = 2(p^2 + p + 1)(p(p - 1) - 1)^3$.

If $\Gamma_R = l_1K_{p(p - 1)}\cup l_2K_{p(p^2 - 1)}$  then, by  \eqref{Eq-1} and \eqref{Eq-2}, we have

\noindent $\sigma_{\snm}(\Gamma_R) = \Big{\{} [-(p(p - 1)-1)^2]^{l_1(p(p - 1) - 1)},  [(p(p - 1)-1)^3]^{l_1}, [-(p(p^2 - 1)-1)^2]^{l_2(p(p^2 - 1) - 1)},$ $ 
[(p(p^2 - 1)-1)^3]^{l_2}\Big{\}}
$
and 
$\E_{\sn}(\Gamma_R) = 2 l_1(p(p - 1) - 1)^3 + 2l_2(p(p^2 - 1) - 1)^3$.

(b) It was shown in \cite[Theorem 2.5]{vr16} that   $\Gamma_R = (p + 1)K_{p^2(p - 1)}$. Therefore, by \eqref{Eq-1} and  \eqref{Eq-2}, we have
\begin{center}
	$\sigma_{\snm}(\Gamma_R) = \Big{\{} [-(p^2(p - 1)-1)^2]^{(p + 1)(p^2(p - 1) - 1)}, \quad [(p^2(p - 1)-1)^3]^{p + 1}\Big{\}}
$
\end{center}
and 
$\E_{\sn}(\Gamma_R) = 2(p + 1)(p^2(p - 1)- 1)^3$.
\end{proof}

\begin{thm}
	If $R$ is a non-commutative ring with unity and $|R| = p^4$  then   $\Gamma_R$ is MSN-integral but not MSN-hyperenergetic.
\end{thm}

\begin{proof}
The fact that  $\Gamma_R$ is MSN-integral follows from the expressions of $\sigma_{\snm}(\Gamma_R)$ given in Theorem \ref{order-p4}.

Since $R$ is  non-commutative with unity having order $p^4$, the order of center of $R$ is not equal to $1$, $p^3$ and $p^4$. Therefore, we consider the following cases.

\noindent \textbf{Case 1.} $|Z(R)| = p$

In this case, $|V(\Gamma_R)| = p^4 - p$. Therefore, $\E_{\sn} (K_{|V(\Gamma_R)|}) = 2(p^4 - p - 1)^3$. By Theorem \ref{order-p4}(a), we have 
\begin{center}
	$\E_{\sn}(\Gamma_R) = 2(p^2 + p + 1)(p^2 - p - 1)^3$ or $2 l_1(p^2 - p - 1)^3 + 2l_2(p^3 - p - 1)^3$, 
\end{center}
where $l_1 + l_2(p + 1) = p^2 + p + 1$. Suppose that $\E_{\sn}(\Gamma_R) = 2(p^2 + p + 1)(p^2 - p - 1)^3$. Since
\begin{align*}
(p^2 + p + 1)(p^2 - p - 1)^3 &= (p^2 + p + 1)(p^2 - p - 1)(p^2 - p - 1)^2\\
&= (p^4 - (p + 1)^2)(p^2 - p - 1)^2\\
&< (p^4 - (p + 1))(p^4 - p - 1)^2\\
&= (p^4 - p - 1)^3,
\end{align*} 
we have $\E_{\sn}(\Gamma_R) < \E_{\sn}(K_{|V(\Gamma_R)|})$.

Suppose that $\E_{\sn}(\Gamma_R) = 2 l_1(p^2 - p - 1)^3 + 2l_2(p^3 - p - 1)^3$, where $l_1 + l_2(p + 1) = p^2 + p + 1$. Note that $|V(\Gamma_R)| = l_1(p^2 - p) + l_2(p^3 - p)$ and so 
\[
\E_{\sn}(K_{|V(\Gamma_R)|}) = 2(l_1(p^2 - p) + l_2(p^3 - p) - 1)^3.
\]
Since $l_1(p^2 - p) \geq 2$ and $l_2(p^3 - p) \geq 6$ we have 
\begin{align*}
(l_1(p^2 - p) + l_2(p^3 - p) - 1)^3 &> l_1^3(p^2 - p)^3 + l_2^3(p^3 - p)^3\\
&> l_1(p^2 - p)^3 + l_2(p^3 - p)^3\\
&> l_1(p^2 - p - 1)^3 + l_2(p^3 - p -1)^3.
\end{align*}
Therefore, $\E_{\sn}(K_{|V(\Gamma_R)|}) > \E_{\sn}(\Gamma_R)$. Hence, $\Gamma_R$ is  not MSN-hyperenergetic. 

\noindent \textbf{Case 2.} $|Z(R)| = p^2$

In this case, $|V(\Gamma_R)| = p^4 - p^2 = (p+1)(p^3 - p^2)$. Therefore, $\E_{\sn} (K_{|V(\Gamma_R)|}) = 2((p+1)(p^3 - p^2) - 1)^3$. By Theorem \ref{order-p4}(b), we have $\E_{\sn}(\Gamma_R) = 2(p + 1)(p^3 - p^2 - 1)^3$. Since
\begin{align*}
(p + 1)(p^3 - p^2 - 1)^3 &< (p + 1)^3(p^3 - p^2 - 1)^3\\
&= ((p+1)(p^3 - p^2 - 1))^3\\
&= ((p+1)(p^3 - p^2) - (p +1))^3\\
&< ((p+1)(p^3 - p^2) - 1)^3,
\end{align*}
it follows that $\E_{\sn}(\Gamma_R) < \E_{\sn} (K_{|V(\Gamma_R)|})$. Hence, $\Gamma_R$ is  not MSN-hyperenergetic.
\end{proof}

\begin{thm}\label{order-p5}
Let $R$ be a  non-commutative ring with unity, $|R| = p^5$  and $Z(R)$ is not a field.
\begin{enumerate}
\item Suppose that  $Z(R)$ has $p^2$ elements. Then 
		
\noindent $\sigma_{\snm}(\Gamma_R) \, = \, \Big{\{} [-(p^3 - p^2-1)^2]^{(p^2 + p + 1)(p^3 - p^2 - 1)}, \,\, [(p^3 - p^2-1)^3]^{p^2 + p + 1}\Big{\}}
$
and 
$\E_{\sn}(\Gamma_R) = 2(p^2 + p + 1)(p^3 - p^2 - 1)^3$		
 or 
$\sigma_{\snm}(\Gamma_R) = \Big{\{} [-(p^3 - p-1)^2]^{l_1(p^3 - p - 1)},$  $[(p^3 - p-1)^3]^{l_1}, \quad [-(p^4 - p^2-1)^2]^{l_2(p^4 - p^2 - 1)}, \quad 
[(p^4 - p^2-1)^3]^{l_2}\Big{\}}
$ \quad
and 

\noindent $\E_{\sn}(\Gamma_R) = 2l_1(p^3 - p - 1)^3 + 2l_2(p^4 - p^2 - 1)^3$ where $l_1 + l_2(p + 1) = p^2 + p + 1$.		

\item Suppose that  $Z(R)$ has $p^3$ elements. Then
\begin{center}
	$\sigma_{\snm}(\Gamma_R) = \Big{\{} [-(p^4 - p^3-1)^2]^{(p + 1)(p^4 - p^3 - 1)}, \quad  [(p^4 - p^3-1)^3]^{p + 1}\Big{\}}
$
\end{center}
and 
$\E_{\sn}(\Gamma_R) = 2(p + 1)(p^4 - p^3 - 1)^3$.	
\end{enumerate}
\end{thm}
\begin{proof}
(a)  It was shown in  \cite[Theorem 2.7]{vr16} that   $\Gamma_R =
(p^2 + p + 1)K_{p^2(p - 1)}$ or $l_1K_{p^2(p - 1)}\cup l_2K_{p^2(p^2 - 1)}$ when $l_1 + l_2(p + 1) = p^2 + p + 1$, if  $|Z(R)| = p^2$. 

If $\Gamma_R = (p^2 + p + 1)K_{p^2(p - 1)}$ then   \eqref{Eq-1} and \eqref{Eq-2} give
\begin{center}
	$\sigma_{\snm}(\Gamma_R) = \Big{\{} [-(p^2(p - 1)-1)^2]^{(p^2 + p + 1)(p^2(p - 1) - 1)}, \,\,\, [(p^2(p - 1)-1)^3]^{p^2 + p + 1}\Big{\}}
$
\end{center}
and 
$\E_{\sn}(\Gamma_R) = 2(p^2 + p + 1)(p^2(p - 1) - 1)^3$.

If $\Gamma_R = l_1K_{p^2(p - 1)}\cup l_2K_{p^2(p^2 - 1)}$ then   \eqref{Eq-1} and \eqref{Eq-2} give

\noindent $\sigma_{\snm}(\Gamma_R) \quad = \quad \Big{\{} [-(p^2(p - 1)-1)^2]^{l_1(p^2(p - 1) - 1)}, \quad  [(p^2(p - 1)-1)^3]^{l_1},$

\qquad\qquad\qquad\qquad\qquad \qquad$ [-(p^2(p^2 - 1)-1)^2]^{l_2(p^2(p^2 - 1) - 1)},  \quad
[(p^2(p^2 - 1)-1)^3]^{l_2}\Big{\}}
$

\noindent and 
$\E_{\sn}(\Gamma_R) = 2l_1(p^2(p - 1) - 1)^3 + 2l_2(p^2(p^2 - 1) - 1)^3$.

(b) It was shown in \cite[Theorem 2.7]{vr16} that $\Gamma_R = (p + 1)K_{p^3(p - 1)}$. Therefore, by \eqref{Eq-1} and \eqref{Eq-2}, we have
\begin{center}
	$\sigma_{\snm}(\Gamma_R) = \Big{\{} [-(p^3(p - 1)-1)^2]^{(p + 1)(p^3(p - 1) - 1)}, \quad  [(p^3(p - 1)-1)^3]^{p + 1}\Big{\}}
$
\end{center}

\noindent and 
$\E_{\sn}(\Gamma_R) = 2(p + 1)(p^3(p - 1) - 1)^3$.
\end{proof}

We conclude this section with the following result.
\begin{thm}
If $R$ is a non-commutative ring with unity and $|R| = p^5$  then   $\Gamma_R$ is MSN-integral but not MSN-hyperenergetic.
\end{thm}

\begin{proof}
The fact that  $\Gamma_R$ is MSN-integral follows from the expressions of $\sigma_{\snm}(\Gamma_R)$ given in Theorem \ref{order-p5}.

Since $R$ is  non-commutative with unity having order $p^5$ and $Z(R)$ is not a field, the order of center of $R$ is not equal to $1$, $p$, $p^4$ and $p^5$. Therefore, we consider the following cases.

\noindent \textbf{Case 1.} $|Z(R)| = p^2$

In this case, $|V(\Gamma_R)| = p^5 - p^2$. Therefore, $\E_{\sn} (K_{|V(\Gamma_R)|}) = 2(p^5 - p^2 - 1)^3$. By Theorem \ref{order-p5}(a), we have 
\begin{center}
$\E_{\sn}(\Gamma_R) = 2(p^2 + p + 1)(p^3 - p^2 - 1)^3$ or $ 2l_1(p^3 - p - 1)^3 + 2l_2(p^4 - p^2 - 1)^3$,
\end{center} where $l_1 + l_2(p + 1) = p^2 + p + 1$. Suppose that 
$\E_{\sn}(\Gamma_R) = 2(p^2 + p + 1)(p^3 - p^2 - 1)^3$. Note that $|V(\Gamma_R)| = (p^2 + p + 1)(p^3 - p^2)$ and so 
\[
\E_{\sn}(K_{|V(\Gamma_R)|}) = 2 ((p^2 + p + 1)(p^3 - p^2) - 1)^3.
\]
Since
\begin{align*}
(p^2 + p + 1)(p^3 - p^2 - 1)^3 &< (p^2 + p + 1)^3(p^3 - p^2 - 1)^3\\
&= ((p^2 + p + 1)(p^3 - p^2) - (p^2 + p + 1))^3\\
&< ((p^2 + p + 1)(p^3 - p^2) - 1)^3,
\end{align*}
we have $\E_{\sn}(\Gamma_R) < \E_{\sn}(K_{|V(\Gamma_R)|})$.

Suppose that $\E_{\sn}(\Gamma_R) = 2l_1(p^3 - p - 1)^3 + 2l_2(p^4 - p^2 - 1)^3$. Note that $|V(\Gamma_R)| = l_1(p^3 - p^2) + l_2(p^4 - p^2)$ and so
\[
\E_{\sn}(K_{|V(\Gamma_R)|}) = 2 (l_1(p^3 - p^2) + l_2(p^4 - p^2) - 1)^3.
\] 
Since $l_1(p^3 - p^2) \geq 4$ and $l_2(p^4 - p^2) \geq 12$ we have 
\begin{align*}
	(l_1(p^3 - p^2) + l_2(p^4 - p^2) - 1)^3 &> l_1^3(p^3 - p^2)^3 + l_2^3(p^4 - p^2)^3\\
	&> l_1(p^3 - p^2)^3 + l_2(p^4 - p^2)^3\\
	&> l_1(p^3 - p^2 - 1)^3 + l_2(p^4 - p^2 -1)^3.
\end{align*}
Therefore, $\E_{\sn}(K_{|V(\Gamma_R)|}) > \E_{\sn}(\Gamma_R)$. Hence, $\Gamma_R$ is  not MSN-hyperenergetic.

\noindent \textbf{Case 2.} $|Z(R)| = p^3$

In this case, $|V(\Gamma_R)| = p^5 - p^3 = (p+1)(p^4 - p^3)$. Therefore, $\E_{\sn} (K_{|V(\Gamma_R)|}) = 2((p+1)(p^4 - p^3) - 1)^3$. By Theorem \ref{order-p5}(b), we have $\E_{\sn}(\Gamma_R) = 2(p + 1)(p^4 - p^3 - 1)^3$. Since
\begin{align*}
	(p + 1)(p^4 - p^3 - 1)^3 &< (p + 1)^3(p^4 - p^3 - 1)^3\\
	&= ((p+1)(p^4 - p^3 - 1))^3\\
	&= ((p+1)(p^4 - p^3) - (p +1))^3\\
	&< ((p+1)(p^4 - p^3) - 1)^3,
\end{align*}
it follows that $\E_{\sn}(\Gamma_R) < \E_{\sn} (K_{|V(\Gamma_R)|})$. Hence, $\Gamma_R$ is  not MSN-hyperenergetic.
\end{proof}

\section{Rings of order $p^2q$ and $p^3q$}
In \cite{vrb14},  Vatandoost et al. obtained the structures of commuting graphs of these two classes of rings. Following them,  various spectra and energies of these graphs are computed in \cite{fn24, fns20}. The genus of commuting graphs of these classes of rings are also computed and determined whether they are planar or toroidal in \cite{fn24-BSPM}. In this section, we consider non-commutative rings of order $p^2q$ and $p^3q$, where $p$ and $q$ are primes, and show that the commuting graphs of these rings are MSN-integral but not MSN-hyperenergetic.  

\begin{thm}\label{order-p^2q}
Let $R$ be a non-commutative ring, $|R| = p^2q$ and $Z(R) = \{0\}$.   
\begin{enumerate}
\item Suppose that ``$t \in \{p, q, p^2, pq\}$ and $(t - 1)$  divides $(p^2q - 1)$". Then 

\noindent $\sigma_{\snm}(\Gamma_R) = \Big{\{} [-(t - 2)^2]^{\frac{(p^2q - 1)(t - 2)}{t - 1}},  [(t - 2)^3]^{\frac{p^2q - 1}{t - 1}}\Big{\}}
$
and 
$\E_{\sn}(\Gamma_R) = \frac{2(p^2q - 1)(t - 2)^3}{t - 1}$.

\item If $(p - 1)l_1 + (q - 1)l_2 + (p^2 - 1)l_3 + (pq - 1)l_4 = p^2q -1$ then 
\begin{align*}
\sigma_{\snm}(\Gamma_R) &= \Big{\{} [-(p -2)^2]^{l_1(p - 2)},  [(p - 2)^3]^{l_1}, [-(q - 2)^2]^{l_2(q - 2)}, [(q - 2)^3]^{l_2},\\
&\qquad[-(p^2 - 2)^2]^{l_3(p^2 - 2)},  [(p^2 - 2)^3]^{l_3}, [-(pq - 2)^2]^{l_4(pq - 2)}, [(pq - 2)^3]^{l_4}\Big{\}}
\end{align*}
and 
$
	\E_{\sn}(\Gamma_R) = 2l_1(p - 2)^3 + 2l_2(q - 2)^3 + 2l_3(p^2-2)^3+ 2l_4(pq -2)^3.
$	
	\end{enumerate}
\end{thm}

\begin{proof}
(a) It was shown in  \cite[Theorem 2.9]{vrb14})  that
$\Gamma_R = \frac{p^2q - 1}{t - 1}K_{t - 1}$. Therefore, by \eqref{Eq-1} and \eqref{Eq-2}, we have
\begin{center}
	$\sigma_{\snm}(\Gamma_R) = \Big{\{} [-(t - 1-1)^2]^{\frac{p^2q - 1}{t - 1} \times (t - 1 - 1)}, \quad  [(t - 1-1)^3]^{\frac{p^2q - 1}{t - 1}}\Big{\}}
$
\end{center}
and 
$\E_{\sn}(\Gamma_R) = 2\times \frac{p^2q - 1}{t - 1}\times(t - 1 - 1)^3$.

(b) It was shown in  \cite[Theorem 2.9]{vrb14})  that $\Gamma_R  = l_1K_{p - 1}\cup l_2K_{q - 1} \cup l_3K_{p^2 - 1} \cup l_4K_{pq - 1}$. Therefore, by \eqref{Eq-1} and \eqref{Eq-2}, we have
\begin{align*}
\sigma_{\snm}(\Gamma_R) = &\Big{\{} [-(p -1-1)^2]^{l_1(p - 1 - 1)}, \quad   [(p - 1-1)^3]^{l_1}, \quad [-(q - 1-1)^2]^{l_2(q - 1 - 1)},\\ 
&\qquad [(q - 1-1)^3]^{l_2}, \quad 	[-(p^2 - 1-1)^2]^{l_3(p^2 - 1 - 1)}, \quad [(p^2 - 1-1)^3]^{l_3}, \\
&\qquad \qquad\qquad\qquad\qquad\quad
	 [-(pq - 1-1)^2]^{l_4(pq - 1 - 1)}, \quad [(pq - 1-1)^3]^{l_4}\Big{\}}
\end{align*}
and 
$
	\E_{\sn}(\Gamma_R) = 2l_1(p - 1 - 1)^3 + 2l_2(q - 1 - 1)^3 + 2l_3(p^2-1 - 1)^3 + 2l_4(pq - 1 - 1)^3.
$
\end{proof}
\begin{thm}
	If $R$ is a non-commutative ring with unity and $|R| = p^2q$  then   $\Gamma_R$ is MSN-integral but not MSN-hyperenergetic.
\end{thm}
\begin{proof}
The fact that  $\Gamma_R$ is MSN-integral follows from the expressions of $\sigma_{\snm}(\Gamma_R)$ given in Theorem \ref{order-p^2q}.

By Theorem \ref{order-p^2q} we also have $\E_{\sn}(\Gamma_R) = \frac{2(p^2q - 1)(t - 2)^3}{t - 1}$ (where $t \in \{p, q, p^2, pq\}$ and $(t - 1)$  divides $(p^2q - 1)$)
or $
\E_{\sn}(\Gamma_R) = 2l_1(p - 2)^3 + 2l_2(q - 2)^3 + 2l_3(p^2-2)^3+ 2l_4(pq -2)^3$,	(where $(p - 1)l_1 + (q - 1)l_2 + (p^2 - 1)l_3 + (pq - 1)l_4 = p^2q -1$). Note that $|V(\Gamma_R)| = p^2q -1$ and so
$\E_{\sn}(K_{|V(\Gamma_R)|}) = 2(p^2q -2)^3$. Suppose that $\E_{\sn}(\Gamma_R) = \frac{2(p^2q - 1)(t - 2)^3}{t - 1}$. Since
\begin{align*}
\frac{(p^2q - 1)(t - 2)^3}{t - 1} &< \left(\frac{p^2q - 1}{t - 1}\right)^3\times (t - 2)^3\\
&= \left(\frac{p^2q - 1}{t - 1}\times (t - 2)\right)^3\\
&= \left(\frac{p^2q - 1}{t - 1}\times (t - 1) - \frac{p^2q - 1}{t - 1}\right)^3\\
&= \left(p^2q - 1 -  \frac{p^2q - 1}{t - 1}\right)^3\\
& < (p^2q - 1 - 1)^3 = (p^2q - 2)^3,
\end{align*}
we have $\E_{\sn}(\Gamma_R) < \E_{\sn}(K_{|V(\Gamma_R)|})$. 

Suppose that $
\E_{\sn}(\Gamma_R) = 2l_1(p - 2)^3 + 2l_2(q - 2)^3 + 2l_3(p^2-2)^3+ 2l_4(pq -2)^3.
$ Note that $(p - 1)l_1 + (q - 1)l_2 + (p^2 - 1)l_3 + (pq - 1)l_4 = p^2q -1$ and so $\E_{\sn}(K_{|V(\Gamma_R)|}) = 2\left((p - 1)l_1 + (q - 1)l_2 + (p^2 - 1)l_3 + (pq - 1)l_4 - 1 \right)^3$. Since $(p - 1)l_1 + (q - 1)l_2 \geq 2$, $(p^2 - 1)l_3 \geq 3$ and $(pq - 1)l_4 \geq 3$ we have 
\begin{align*}
((p - 1)l_1 + (q - 1)l_2 &+ (p^2 - 1)l_3 + (pq - 1)l_4 - 1)^3\\
& > ((p - 1)l_1 + (q - 1)l_2)^3  + (p^2 - 1)^3l_3^3 +  (pq - 1)^3l_4^3 \\
& > (p - 1)^3l_1^3 + (q - 1)^3l_2^3  + (p^2 - 1)^3l_3^3 +  (pq - 1)^3l_4^3 \\
& > l_1(p - 1)^3 + l_2(q - 1)^3  + l_3(p^2 - 1)^3 +  l_4(pq - 1)^3 \\
& > l_1(p - 2)^3 + l_2(q - 2)^3  + l_3(p^2 - 2)^3 +  l_4(pq - 2)^3.
\end{align*}
Therefore, $\E_{\sn}(K_{|V(\Gamma_R)|}) > \E_{\sn}(\Gamma_R)$. Hence,  $\Gamma_R$ is not MSN-hyperenergetic. 
\end{proof}

\begin{thm}\label{order-p^3q-pq}
Let $R$ be a ring   with  unity, $|R| = p^3q$ and $|Z(R)| = pq$. Then $\sigma_{\snm}(\Gamma_R) = \Big{\{} [-(p^2q - pq-1)^2]^{(p + 1)(p^2q - pq - 1)}, \quad [(p^2q - pq-1)^3]^{p + 1}\Big{\}}
$
and 
$\E_{\sn}(\Gamma_R) = 2(p + 1)(p^2q - pq - 1)^3$.
\end{thm}
\begin{proof}
It was shown in \cite[Theorem 2.12]{vrb14}) that $\Gamma_R = (p + 1)K_{pq(p - 1)}$. Therefore, by \eqref{Eq-1}, we have
$\sigma_{\snm}(\Gamma_R) = \Big{\{} [-(pq(p - 1)-1)^2]^{(p + 1)(pq(p - 1) - 1)},  [(pq(p - 1)-1)^3]^{p + 1}\Big{\}}
$
and 
$\E_{\sn}(\Gamma_R) = 2(p + 1)(pq(p - 1) - 1)^3$.
\end{proof}

\begin{thm}\label{order-p^3q}
Let $R$ be a ring   with  unity, $|R| = p^3q$ and $|Z(R)| = p^2$. 
\begin{enumerate}		
\item If \, $(p - 1) \mid (pq - 1)$ \, then  \,
$\sigma_{\snm}(\Gamma_R) = \Big{\{} [-(p^3 - p^2-1)^2]^{\frac{(pq - 1)(p^3 - p^2 - 1)}{p - 1}},$ $[(p^3 - p^2-1)^3]^{\frac{pq - 1}{p - 1}}\Big{\}}
		$
		and 
		$\E_{\sn}(\Gamma_R) = \frac{2(pq - 1)(p^3 - p^2 - 1)^3}{p - 1}$.
		
\item If $(q - 1)\mid (pq - 1)$ \, then  \,
$\sigma_{\snm}(\Gamma_R) \,\, = \,\, \Big{\{} [-(p^2q - p^2-1)^2]^{\frac{(pq - 1)(p^2q - p^2 - 1)}{q - 1}},$  $ [(p^2q - p^2-1)^3]^{\frac{pq - 1}{q - 1}}\Big{\}}
	$
	and 
	$\E_{\sn}(\Gamma_R) = \frac{2(pq - 1)(p^2q - p^2 - 1)^3}{q - 1}$.	
		
\item If $pq - 1 = (p - 1)l_1 + (q - 1)l_2$ then 
$\sigma_{\snm}(\Gamma_R) = \Big{\{} [-(p^3 - p^2-1)^2]^{l_1(p^3 - p^2 - 1)},$ $[(p^3 - p^2-1)^3]^{l_1}, \quad [-(p^2q - p^2-1)^2]^{l_2(p^2q - p^2 - 1)}, \quad 
[(p^2q - p^2-1)^3]^{l_2}\Big{\}}
$
and 
$\E_{\sn}(\Gamma_R) = 2l_1(p^3 - p^2 - 1)^3 + 2l_2(p^2q - p^2 - 1)^3$.
\end{enumerate}
\end{thm}
\begin{proof}
(a) It was shown in \cite[Theorem 2.12]{vrb14}) that $\Gamma_R = \frac{pq - 1}{p - 1}K_{p^2(p - 1)}$. Therefore, by \eqref{Eq-1}, we have
$\sigma_{\snm}(\Gamma_R) = \Big{\{} [-(p^2(p - 1)-1)^2]^{\frac{pq - 1}{p - 1}\times(p^2(p - 1) - 1)},  [(p^2(p - 1)-1)^3]^{\frac{pq - 1}{p - 1}}\Big{\}}
$
and 
$\E_{\sn}(\Gamma_R) = 2\times\frac{pq - 1}{p - 1}\times(p^2(p - 1) - 1)^3$.

(b) It was shown in \cite[Theorem 2.12]{vrb14}) that $\Gamma_R =  \frac{pq - 1}{q - 1}K_{p^2(q - 1)}$. Therefore, by \eqref{Eq-1} and \eqref{Eq-2}, we have
$\sigma_{\snm}(\Gamma_R) = \Big{\{} [-(p^2(q - 1)-1)^2]^{\frac{pq - 1}{q - 1}\times(p^2(q - 1) - 1)}, \quad  [(p^2(q - 1)-1)^3]^{\frac{pq - 1}{q - 1}}\Big{\}}
$
and 
$\E_{\sn}(\Gamma_R) = 2\times\frac{pq - 1}{q - 1}\times(p^2(q - 1) - 1)^3$.

(c) It was shown in \cite[Theorem 2.12]{vrb14}) that $\Gamma_R = l_1K_{p^2(p - 1)}\cup l_2K_{p^2(q - 1)}$.  Therefore, by \eqref{Eq-1} and \eqref{Eq-2}, we have

$\sigma_{\snm}(\Gamma_R) =  \Big{\{} [-(p^2(p - 1)-1)^2]^{l_1(p^2(p - 1) - 1)}, \quad [(p^2(p - 1)-1)^3]^{l_1},$ 

\qquad\qquad\qquad\qquad\qquad\qquad$[-(p^2(q - 1)-1)^2]^{l_2(p^2(q - 1) - 1)}, \quad
[(p^2(q - 1)-1)^3]^{l_2}\Big{\}}
$

and 
$\E_{\sn}(\Gamma_R) = 2l_1(p^2(p - 1) - 1)^3 + 2l_2(p^2(q - 1) - 1)^3$.
\end{proof}

We conclude this section with the following result.
\begin{thm}
If $R$ is a non-commutative ring with unity such that $|R| = p^3q$ and $|Z(R)|$ is not a prime  then   $\Gamma_R$ is MSN-integral but not MSN-hyperenergetic.
\end{thm}
\begin{proof}
If $R$ is a non-commutative ring with unity such that $|R| = p^3q$ and $|Z(R)|$ is not a prime  then $|Z(R)| = pq$ or $p^2$. Therefore,
the fact that  $\Gamma_R$ is MSN-integral follows from the expressions of $\sigma_{\snm}(\Gamma_R)$ given in Theorem \ref{order-p^3q-pq} and Theorem \ref{order-p^3q}.

If $|Z(R)| = pq$ then $|V(\Gamma_R)| = p^3q - pq$. Therefore, $\E_{\sn}(K_{|V(\Gamma_R)|}) = 2(p^3q - pq -1)^3$. By Theorem \ref{order-p^3q-pq} we have
$\E_{\sn}(\Gamma_R) = 2((p + 1)(p^2q - pq - 1)^3$. Note that
\begin{align*}
((p + 1)(p^2q - pq - 1)^3 &< ((p + 1)^3(p^2q - pq - 1)^3\\
&= ((p + 1)(p^2q - pq - 1))^3\\
&= ((p + 1)(p^2q - pq) - (p+1))^3\\
&= (p^3q - pq - (p+1))^3
< (p^3q - pq - 1)^3.
\end{align*}
Therefore, $\E_{\sn}(\Gamma_R) < \E_{\sn}(K_{|V(\Gamma_R)|})$.

If $|Z(R)| = p^2$ then $|V(\Gamma_R)| = p^3q - p^2$. Therefore, $\E_{\sn}(K_{|V(\Gamma_R)|}) = 2(p^3q - p^2 -1)^3$. By Theorem \ref{order-p^3q} we have $\E_{\sn}(\Gamma_R) = \frac{2(pq - 1)(p^2t - p^2 - 1)^3}{t - 1}$ (where $t \in \{p, q\}$ and $(t - 1)$ divides $(pq - 1)$)
or $\E_{\sn}(\Gamma_R) = 2l_1(p^3 - p^2 - 1)^3 + 2l_2(p^2q - p^2 - 1)^3$ (where $pq - 1 = (p - 1)l_1 + (q - 1)l_2$).

Thus, if $t \in \{p, q\}$ and $t - 1$ divides $pq - 1$ then we have
\begin{align*}
	\frac{(pq - 1)(p^2t - p^2 - 1)^3}{t - 1} &< \left(\frac{pq - 1}{t - 1}\right)^3(p^2t - p^2 - 1)^3\\
	&=\left(\frac{pq - 1}{t - 1}\times(p^2t - p^2 - 1)\right)^3\\
	&=\left(\frac{pq - 1}{t - 1}\times(p^2t - p^2) - \frac{pq - 1}{t - 1}\right)^3\\
	&= \left(p^3q - p^2 - \frac{pq - 1}{t - 1}\right)^3 < (p^3q - p^2 - 1)^3.
\end{align*}
Therefore, $\E_{\sn}(\Gamma_R) < \E_{\sn}(K_{|V(\Gamma_R)|})$.

Suppose that $\E_{\sn}(\Gamma_R) = 2l_1(p^3 - p^2 - 1)^3 + 2l_2(p^2q - p^2 - 1)^3$. Note that $|V(\Gamma_R)| = l_1(p^3 - p^2) + l_2(p^2q - p^2)$ and so $\E_{\sn}(K_{|V(\Gamma_R)|}) = 2(l_1(p^3 - p^2) + l_2(p^2q - p^2) - 1)^3$. Since  $l_1(p^3 - p^2) \geq 4$ and $l_2(p^2q - p^2) \geq 4$ we have
\begin{align*}
(l_1(p^3 - p^2) + l_2(p^2q - p^2) - 1)^3 &> l_1^3(p^3 - p^2)^3 + l_2^3(p^2q - p^2)^3\\
&> l_1(p^3 - p^2)^3 + l_2(p^2q - p^2)^3\\
&> l_1(p^3 - p^2 - 1)^3 + l_2(p^2q - p^2 -1)^3.
\end{align*}
Therefore, $\E_{\sn}(K_{|V(\Gamma_R)|}) > \E_{\sn}(\Gamma_R)$. 
Hence, $\Gamma_R$ is  not MSN-hyperenergetic.
\end{proof}

\section{Conclusion}
A non-commutative ring $R$ is called a commutative centralizer ring (for short CC-ring) if all the centralizers of non-central elements of $R$ are commutative. Interestingly, all the rings considered in Section 2-4 are  CC-rings and their commuting graphs are MSN-integral but not MSN-hyperenergetic. Therefore, the following conjecture is natural.
\begin{conj}
If $R$ is a finite non-commutative CC-ring then $\Gamma_R$ is MSN-integral but not MSN-hyperenergetic.
\end{conj}

The following result proves the above conjecture.  
\begin{thm}\label{cc-ring-energies}
If $S_1, S_2, \ldots, S_n$ are the non-identical centralizers  of  $s \in R \setminus~Z(R)$, where $R$ is  a finite CC-ring and $|Z(R)| = m$,
	then

\noindent $\sigma_{\snm}(\Gamma_R) = \Big{\{} [-(|S_1| - m -1)^2]^{|S_1| - m - 1},\, [(|S_1| - m - 1)^3]^1, \, [-(|S_2| - m -1)^2]^{|S_2| - m - 1},$ 

\qquad\qquad$[(|S_2| - m - 1)^3]^1, \dots, [-(|S_n| - m -1)^2]^{|S_n| - m - 1},\,\, [(|S_n| - m - 1)^3]^1,\Big{\}}$  

\noindent and $\E_{\sn}(\Gamma_R) = 2 \displaystyle{\sum_{i = 1}^n} (|S_i| - m - 1)^3$. Further, $\E_{\sn}(\Gamma_R) < \E_{\sn}(K_{|R| - m})$.
%
\end{thm}
\begin{proof}
It was shown in \cite[Theorem 2.1]{Dutta-Nath-CCR} that $\Gamma_R = K_{|S_1| - m} \cup K_{|S_2| - m} \cup \cdots \cup K_{|S_n| - m} $. Therefore, by \eqref{Eq-1} and \eqref{Eq-2}, we have

\noindent $\sigma_{\snm}(\Gamma_R) = \Big{\{} [-(|S_1| - m -1)^2]^{|S_1| - m - 1},\, [(|S_1| - m - 1)^3]^1, \, [-(|S_2| - m -1)^2]^{|S_2| - m - 1},$ 

\qquad\qquad\quad$[(|S_2| - m - 1)^3]^1, \dots, [-(|S_n| - m -1)^2]^{|S_n| - m - 1},\,\, [(|S_n| - m - 1)^3]^1,\Big{\}}$  

\noindent and $\E_{\sn}(\Gamma_R) = 2 \displaystyle{\sum_{i = 1}^n} (|S_i| - m - 1)^3$.
%
 
We have
\begin{align*}
\E_{\sn}(K_{|R| - m}) &= 2(|R| - m - 1)^3 \\
&= 2(|S_1| - m + |S_2| - m + \cdots + |S_n| - m  - 1)^3\\
&> 2(|S_1| - m -1 + |S_2| - m -1 + \cdots + |S_n| - m  - 1)^3\\
&> 2(|S_1| - m -1)^3 + (|S_2| - m -1)^3 + \cdots + (|S_n| - m  - 1)^3.
\end{align*}
Hence, the result follows.
\end{proof}
More generally,  we have proved the following result.
\begin{thm}
If $\Gamma = l_1K_{m_1} \cup l_2K_{m_2} \cup \cdots \cup l_rK_{m_r}$ then $\Gamma$ is MSN-integral and $\E_{\sn}(K_{|V(\Gamma)|}) >  \E_{\sn}(\Gamma)$.   
\end{thm} 
\begin{proof}
The fact that `$\Gamma$ is MSN-integral' is clear from \eqref{Eq-1}. Note that $|V(\Gamma)| = l_1m_1 + l_2m_2 + \cdots + l_rm_r$ and so $\E_{\sn}(K_{|V(\Gamma)|}) = 2(l_1m_1 + l_2m_2 + \cdots + l_rm_r - 1)^3$. We have
\begin{align*}
(l_1m_1 + l_2m_2 + \cdots + l_rm_r - 1)^3 &\geq (l_1m_1 + l_2m_2 + \cdots + l_rm_r - l_r)^3\\
& > (l_1m_1)^3 + (l_2m_2)^3 + \cdots + (l_r^3(m_r - 1))^3\\
& > l_1m_1^3 + l_2m_2^3 + \cdots + l_r^3(m_r - 1)^3\\
& > l_1(m_1 - 1)^3 + l_2(m_2 - 1)^3 + \cdots + l_r^3(m_r - 1)^3.
\end{align*}
Hence, the result follows from \eqref{Eq-2}.
\end{proof}

The above arguments motivate us to prove Conjecture \ref{Conj-1.1} and Conjecture \ref{Conj-1.2}. 
The following result proves Conjecture \ref{Conj-1.1}.
\begin{thm}\label{Conj-3-NFDS-2021}
If $\Gamma = l_1K_{m_1} \cup l_2K_{m_2} \cup \cdots \cup l_rK_{m_r}$, where $1 \leq m_1 < m_2 < \cdots < m_r$, then $\Gamma$ is $\Ecn(K_{|V(\Gamma)|}) >  \Ecn(\Gamma)$.
\end{thm}
\begin{proof}
We have $|V(\Gamma)| = \sum_{i = 1}^rl_im_i$ and so (by  \cite[Theorem 2]{NFDS-2021})
\[
\Ecn(K_{|V(\Gamma)|}) = 2\left(\sum_{i = 1}^rl_im_i-1\right)\left(\sum_{i = 1}^rl_im_i-2\right).
\] 
Also, $\Ecn(\Gamma) = 2\sum_{i=1}^rl_i(m_i-1)(m_i-2)$. Therefore, if $m_1 = 1$ then 
$$\Ecn(\Gamma) = 2\sum_{i=2}^rl_i(m_i-1)(m_i-2).
$$
We have
\begin{align*}
\left(\sum_{i = 1}^rl_im_i-1\right)\left(\sum_{i = 1}^rl_im_i-2\right) &> \left(\sum_{i = 2 }^rl_im_i-1\right)\left(\sum_{i = 2}^rl_im_i-2\right)\\
&\geq \left(\sum_{i = 2 }^rl_im_i-l_i\right)\left(\sum_{i = 2}^rl_im_i-2l_i\right)\\
&= \left(\sum_{i = 2 }^rl_i(m_i-1)\right)\left(\sum_{i = 2}^rl_i(m_i-2)\right)\\
& > \sum_{i = 2 }^rl_i^2(m_i-1)(m_i-2)\\
& \geq \sum_{i = 2 }^rl_i(m_i-1)(m_i-2).
\end{align*}
If $m_1 > 1$ then
\begin{align*}
	\left(\sum_{i = 1}^rl_im_i-1\right)\left(\sum_{i = 1}^rl_im_i-2\right) &\geq \left(\sum_{i = 1 }^rl_im_i-l_i\right)\left(\sum_{i = 1}^rl_im_i-2l_i\right)\\
	&= \left(\sum_{i = 1 }^rl_i(m_i-1)\right)\left(\sum_{i = 1}^rl_i(m_i-2)\right)\\
	& > \sum_{i = 1 }^rl_i^2(m_i-1)(m_i-2)\\
	& \geq \sum_{i = 1 }^rl_i(m_i-1)(m_i-2).
\end{align*}
Hence, $\Ecn(K_{|V(\Gamma)|}) > \Ecn(\Gamma)$ in both the cases.
\end{proof}
The following corollary proves Conjecture \ref{Conj-1.2}.
\begin{cor}
If  $R$ is  a finite CC-ring then $\Ecn(K_{|V(\Gamma_R)|}) > \Ecn(\Gamma_R)$.
\end{cor}
\begin{proof}
The result follows from Theorem \ref{Conj-3-NFDS-2021} noting that $\Gamma_R$ is of the form $l_1K_{m_1} \cup l_2K_{m_2} \cup \cdots \cup l_rK_{m_r}$ for some $l_i$ and $m_i$.
\end{proof}
We conclude this paper with the following problems.
\begin{prob}
Characterize all finite non-commutative rings $R$  such that $\Gamma_R$ is MSN-integral.
\end{prob}
\begin{prob}
Give examples of finite non-commutative rings $R$ such that $\Gamma_R$ is MSN-hyperenergetic. Also, characterize all finite non-commutaive rings $R$ (if exist) such that $\Gamma_R$ is MSN-hyperenergetic.
\end{prob}


\end{document}